\documentclass[reqno,11pt]{amsart}

\usepackage{a4wide}
\usepackage{color}
\usepackage{mathtools}
\usepackage{amsmath}
\usepackage{amssymb}
\numberwithin{equation}{section}
\usepackage[colorlinks,citecolor=green,linkcolor=red]{hyperref}

\usepackage[latin1]{inputenc}

\newcommand{\nchi}{{\raise.3ex\hbox{\(\chi\)}}}
\newcommand{\N}{\mathbb{N}}
\newcommand{\R}{\mathbb{R}}
\newcommand{\E}{{\mathcal E}}
\newcommand{\Ch}{{\rm Ch}}
\newcommand{\sfd}{{\sf d}}
\renewcommand{\d}{{\mathrm d}}
\newcommand{\X}{{\rm X}}
\newcommand{\mm}{\mathfrak{m}}
\newcommand{\LIP}{{\rm LIP}}
\newcommand{\Lip}{{\rm Lip}}
\newcommand{\lip}{{\rm lip}}
\newcommand{\limi}{\varliminf}
\newcommand{\lims}{\varlimsup}
\newcommand{\eps}{\varepsilon}

\newcommand{\fr}{\penalty-20\null\hfill\(\blacksquare\)}

\newtheorem{theorem}{Theorem}[section]
\newtheorem{corollary}[theorem]{Corollary}
\newtheorem{lemma}[theorem]{Lemma}
\newtheorem{proposition}[theorem]{Proposition}

\newtheorem{example}[theorem]{Example}
\newtheorem{remark}[theorem]{Remark}

\title{Gamma-convergence of Cheeger energies with respect to increasing distances}
\author{Danka Lu\v{c}i\'{c}}
\address[Danka Lu\v{c}i\'{c}]{Universit\`{a} di Pisa, Dipartimento di Matematica,
Largo Bruno Pontecorvo 5, 56127 Pisa, Italy}
\email{danka.lucic@dm.unipi.it}

\author{Enrico Pasqualetto}
\address[Enrico Pasqualetto]{Scuola Normale Superiore, Piazza dei Cavalieri 7,
56126 Pisa, Italy}
\email{enrico.pasqualetto@sns.it}
\begin{document}
\date{\today} 
\begin{abstract}
We prove a \(\Gamma\)-convergence result for Cheeger energies
along sequences of metric measure spaces, where the measure space
is kept fixed, while distances are monotonically converging from
below to the limit one. As a consequence, we show that the
infinitesimal Hilbertianity condition is stable under
this kind of convergence of metric measure spaces.
\end{abstract}
\maketitle
\section{Introduction}
In the successful theory of weakly differentiable functions
over metric measure spaces, a leading role is played by the
so-called Cheeger \(p\)-energy, which was introduced in
\cite{Cheeger00} and generalises the classical Dirichlet
\(p\)-energy functional. The purpose of this paper is
to study the convergence of Cheeger \(p\)-energies along a
sequence of metric measure spaces, where the underlying set
and the measure are fixed, while distances monotonically converge
from below.

More precisely, given a metric measure space \((\X,\sfd,\mm)\)
and a sequence \((\sfd_i)_{i\in\N}\) of distances on \(\X\)
inducing the same topology as \(\sfd\) and satisfying
\(\sfd_i\nearrow\sfd\), we prove (in Theorem
\ref{thm:main_Mosco-conv}) that for any \(p\in(1,\infty)\)
the Cheeger \(p\)-energies
\(\E_{\Ch,p}^{\sfd_i}\colon L^p(\mm)\to[0,+\infty]\)
associated with \((\X,\sfd_i,\mm)\) converge to
\(\E_{\Ch,p}^\sfd\) in the sense of Mosco. As shown in
Example \ref{ex:failure_conv_from_above}, this kind of
statement might totally fail in the case where \(\sfd_i\searrow\sfd\).
Since the family of quadratic forms is closed under Mosco-convergence,
an interesting consequence of Theorem \ref{thm:main_Mosco-conv} is the stability
of the infinitesimal Hilbertianity condition (that was introduced in
\cite{Gigli12} and states the quadraticity of the Cheeger
\(2\)-energy functional) with respect to increasing limits of
the involved distances.

Sub-Riemannian manifolds constitute a significant example of metric
structures where the above results apply, as  the induced length
distances can be monotonically approximated from below by
Riemannian ones; cf.\ the discussion in Remark \ref{rmk:sub-Riem}.
\medskip

A previous result on 
the Mosco-convergence of Cheeger energies was
obtained in \cite[Theorem 6.8]{GigliMondinoSavare13} for
sequences of \({\sf CD}(K,\infty)\) spaces that converge with
respect to (a variant of) the pointed measured
Gromov--Hausdorff topology. However, since measured
Gromov--Hausdorff convergence is a zeroth-order concept,
while the Cheeger energy is a first-order one, we cannot
expect such Mosco-convergence result to hold on arbitrary
metric measure spaces. Indeed, given an arbitrary metric measure
space \((\X,\sfd,\mm)\), one can easily construct a sequence of
discrete measures \((\mm_i)_{i\in\N}\) that weakly converge
to \(\mm\); consequently, since the Cheeger energies associated
with the spaces \((\X,\sfd,\mm_i)\) are identically zero, the
Mosco-convergence result will generally fail. In the case
of \({\sf CD}(K,\infty)\) spaces, the convergence of the
Cheeger energies is boosted by the uniform lower bound on
the Ricci curvature (encoded in the \(\sf CD\) condition),
which is a second-order notion. Conversely, in our main
Theorem \ref{thm:main_Mosco-conv} we do not require any
regularity at the level of the involved metric measure spaces,
but instead we consider a notion of convergence that is much
stronger than the pointed measured Gromov--Hausdorff one.
\medskip

We conclude the introduction by briefly describing an
approximation result for Lipschitz functions (Proposition
\ref{prop:approx_lip_a}) that will have an essential role
in the proof of Theorem \ref{thm:main_Mosco-conv}. Under
the same assumptions as in the Mosco-convergence result for
Cheeger energies, we prove that every \(\sfd\)-Lipschitz function
\(f\) can be approximated (in the integral sense)
by a \(\sfd_i\)-Lipschitz function \(g\), for some
index \(i\in\N\) sufficiently large, such that the integral
of the \(p\)-power of the asymptotic slope of \(g\) is close
to that of \(f\). This goal is achieved by appealing to the
asymptotic-slope-preserving extension result for Lipschitz
functions  obtained in \cite{DMGP20}.
\medskip

\textbf{Acknowledgements.}
The authors wish to thank Tapio Rajala for the careful reading of
a preliminary version of this paper. The first named author was
supported by the project 2017TEXA3H ``Gradient flows, Optimal
Transport and Metric Measure Structures", funded by the Italian
Ministry of Research and University. The second named author was
supported by the Balzan project led by Luigi Ambrosio.
\section{Preliminaries}
Let \((\X,\sfd)\) be a given metric space. We denote by \(\tau(\sfd)\) the topology
on \(\X\) induced by the distance \(\sfd\). The open ball and the closed ball
of center \(x\in\X\) and radius \(r>0\) are given by
\[
B_r^\sfd(x)\coloneqq\big\{y\in\X\;\big|\;\sfd(x,y)<r\big\},\qquad
\bar B_r^\sfd(x)\coloneqq\big\{y\in\X\;\big|\;\sfd(x,y)\leq r\big\},
\]
respectively. The space of \(\sfd\)-Lipschitz functions \(f\colon\X\to\R\) will be
denoted by \(\LIP_\sfd(\X)\). Given any \(f\in\LIP_\sfd(\X)\) and
\(E\subseteq\X\), we denote by \(\Lip_\sfd(f;E)\in[0,+\infty)\)
and \(\lip_a^\sfd(f)\colon\X\to[0,+\infty)\) the Lipschitz constant of
\(f|_E\) and the asymptotic slope of \(f\), respectively. Videlicet, we set
\[\begin{split}
\Lip_\sfd(f;E)&\coloneqq\sup\bigg\{\frac{\big|f(x)-f(y)\big|}{\sfd(x,y)}
\;\bigg|\;x,y\in E,\,x\neq y\bigg\},\\
\lip_a^\sfd(f)(x)&\coloneqq\inf_{r>0}\Lip_\sfd\big(f;B_r^\sfd(x)\big),
\quad\text{ for every }x\in\X,
\end{split}\]
where we adopt the convention that \(\Lip_\sfd(f;\emptyset)=\Lip_\sfd\big(f;\{x\}\big)
\coloneqq 0\). For the sake of brevity, we will use the shorthand notation
\(\Lip_\sfd(f)\coloneqq\Lip_\sfd(f;\X)\). Observe that \(\lip_a^\sfd(f)\leq\Lip_\sfd(f)\).
\begin{remark}\label{rmk:ineq_dist}{\rm
Let \(\X\) be a non-empty set. Let \(\sfd\) and \(\sfd'\) be distances on \(\X\)
such that \(\sfd\leq\sfd'\). Then for any \(x,x',y,y'\in\X\) it holds that
\[
\big|\sfd(x,y)-\sfd(x',y')\big|\leq\sfd(x,x')+\sfd(y,y')\leq
\sfd'(x,x')+\sfd'(y,y')\leq\sqrt 2\,(\sfd'\times\sfd')\big((x,y),(x',y')\big),
\]
thus \(\sfd\colon\X\times\X\to[0,+\infty)\) is \((\sfd'\times\sfd')\)-continuous,
where \(\sfd'\times\sfd'\) stands for the product distance
\[
(\sfd'\times\sfd')\big((x,y),(x',y')\big)\coloneqq
\sqrt{\sfd'(x,x')^2+\sfd'(y,y')^2},\quad\text{ for every }(x,y),(x',y')\in\X\times\X.
\]
Moreover, given \(f\in\LIP_\sfd(\X)\) and \(E\subseteq\X\), we can estimate
\(\big|f(x)-f(y)\big|\leq\Lip_\sfd(f;E)\,\sfd'(x,y)\) for every \(x,y\in E\).
This shows that \(\LIP_\sfd(\X)\subseteq\LIP_{\sfd'}(\X)\) and that
\[
\Lip_{\sfd'}(f;E)\leq\Lip_\sfd(f;E),\quad\text{ for every }
f\in\LIP_\sfd(\X)\text{ and }E\subseteq\X.
\]
In particular, we obtain that \(\lip_a^{\sfd'}(f)\leq\lip_a^\sfd(f)\)
for every \(f\in\LIP_\sfd(\X)\).
\fr}\end{remark}

By a metric measure space \((\X,\sfd,\mm)\) we mean a complete and separable
metric space \((\X,\sfd)\), which is endowed with a boundedly-finite Borel
measure \(\mm\geq 0\). One of the possible ways to introduce Sobolev spaces
on \((\X,\sfd,\mm)\) is via relaxation of upper gradients. Instead of the original approach that
was introduced by Cheeger \cite{Cheeger00}, we present its
equivalent reformulation (via relaxation of the asymptotic slope)
that was studied by Ambrosio--Gigli--Savar\'{e} in
\cite{AmbrosioGigliSavare11-3}.
\medskip

Given a metric measure space \((\X,\sfd,\mm)\) and an exponent \(p\in(1,\infty)\),
let us define the asymptotic \(p\)-energy functional \(\E_{a,p}^\sfd\colon
L^p(\mm)\to[0,+\infty]\) as
\[
\E_{a,p}^\sfd(f)\coloneqq\left\{\begin{array}{ll}
\frac{1}{p}\int\lip_a^\sfd(f)^p\,\d\mm,\\
+\infty,
\end{array}\quad\begin{array}{ll}
\text{ if }f\in\LIP_\sfd(\X)\text{ is boundedly-supported,}\\
\text{ otherwise.}
\end{array}\right.
\]
Then the Cheeger \(p\)-energy functional
\(\E_{\Ch,p}^\sfd\colon L^p(\mm)\to[0,+\infty]\) is defined
as the \(L^p(\mm)\)-lower semicontinuous envelope of \(\E_{a,p}^\sfd\).
Videlicet, for any function \(f\in L^p(\mm)\) we define
\[
\E_{\Ch,p}^\sfd(f)\coloneqq\inf\bigg\{\limi_{n\to\infty}\E_{a,p}^\sfd(f_n)
\;\bigg|\;(f_n)_n\subseteq L^p(\mm),\,f_n\to f
\text{ strongly in }L^p(\mm)\bigg\}.
\]
It turns out that \(\E_{\Ch,p}^\sfd\) is weakly lower semicontinuous,
meaning that \(\E_{\Ch,p}^\sfd(f)\leq\limi_n\E_{\Ch,p}^\sfd(f_n)\) whenever
\(f\in L^p(\mm)\) and \((f_n)_n\subseteq L^p(\mm)\) satisfy \(f_n\rightharpoonup f\)
weakly in \(L^p(\mm)\). The \(p\)-Sobolev space on \((\X,\sfd,\mm)\) is then defined as
the finiteness domain of \(\E_{\Ch,p}^\sfd\), videlicet
\[
W^{1,p}(\X)\coloneqq\big\{f\in L^p(\mm)\;\big|\;\E_{\Ch,p}^\sfd(f)<+\infty\big\}.
\]
It holds that \(W^{1,p}(\X)\) is a Banach space if endowed with
the following norm:
\[
\|f\|_{W^{1,p}(\X)}\coloneqq\big(\|f\|_{L^p(\mm)}^p+p\,\E_{\Ch,p}^\sfd(f)
\big)^{1/p},\quad\text{ for every }f\in W^{1,p}(\X).
\]
In general, the \(2\)-Sobolev space is not Hilbert. A metric measure space
\((\X,\sfd,\mm)\) is said to be infinitesimally Hilbertian
\cite{Gigli12} provided the associated \(2\)-Sobolev space
\(W^{1,2}(\X)\) is Hilbert, or equivalently provided
\(\E_{\Ch,2}^\sfd\) is a quadratic form.
\begin{remark}\label{rmk:ineq_ener}{\rm
Let \((\X,\sfd,\mm)\) be a metric measure space. Let \(\sfd'\) be a distance
on \(\X\) with \(\sfd\leq\sfd'\) and \(\tau(\sfd)=\tau(\sfd')\), thus
\((\X,\sfd',\mm)\) is a metric measure space as well. Then Remark
\ref{rmk:ineq_dist} yields
\begin{equation}\label{eq:ineq_ener}
\E_{a,p}^{\sfd'}\leq\E_{a,p}^\sfd,\qquad\E_{\Ch,p}^{\sfd'}\leq\E_{\Ch,p}^\sfd,
\end{equation}
for any given exponent \(p\in(1,\infty)\).
\fr}\end{remark}
\section{An approximation result}
Aim of this section is to achieve an approximation
result for Lipschitz functions (\emph{i.e.}, Proposition \ref{prop:approx_lip_a}),
which will be a key tool in order to prove our main
Theorem \ref{thm:main_Mosco-conv}.
\begin{remark}\label{rmk:unif_conv_cpt}{\rm
Let \(\X\) be a non-empty set and \((\sfd_i)_{i\in\bar\N}\)
a sequence of distances on \(\X\) satisfying
\[
\sfd_i(x,y)\nearrow\sfd_\infty(x,y),\quad\text{ for every }x,y\in\X.
\]
Then \(\sfd_i\to\sfd_\infty\) uniformly on each subset of \(\X\times\X\)
that is compact with respect to \(\tau(\sfd_\infty\times\sfd_\infty)\).
Indeed, Remark \ref{rmk:ineq_dist} grants that \(\sfd_i\colon\X\times\X\to\R\)
is \((\sfd_\infty\times\sfd_\infty)\)-continuous for all \(i\in\bar\N\).
\fr}\end{remark}
We begin with a preliminary approximation result, where the
given Lipschitz function is uniformly approximated on a compact
set and just the global Lipschitz constant is controlled.
\begin{lemma}\label{lem:approx_Lip}
Let \((\X,\sfd)\) be a metric space. Suppose there exists a sequence
\((\sfd_i)_{i\in\N}\) of distances on \(\X\) such that \(\sfd_i(x,y)\nearrow
\sfd(x,y)\) as \(i\to\infty\) for every \(x,y\in\X\). Let \(f\in\LIP_\sfd(\X)\)
be given. Then for any \(K\subseteq\X\) compact and \(\eps>0\) there exist
\(i\in\N\) and \(g\in\LIP_{\sfd_i}(\X)\) such that
\begin{subequations}\begin{align}\label{eq:approx_Lip_cl1}
\max_K|g-f|&\leq\eps,\\
\label{eq:approx_Lip_cl2}
\Lip_{\sfd_i}(g)&\leq\Lip_\sfd(f).
\end{align}\end{subequations}
\end{lemma}
\begin{proof}
Call \(L\coloneqq\Lip_\sfd(f)\) and fix a dense sequence \((x_j)_{j\in\N}\)
in \(K\). Given any \(n\in\N\), we define
\[
\tilde g_n(x)\coloneqq\big(-L\,\sfd(x,x_1)+f(x_1)\big)\vee\dots\vee
\big(-L\,\sfd(x,x_n)+f(x_n)\big)-\frac{1}{n},\quad\text{ for every }x\in\X.
\]
Note that \((\tilde g_n)_{n\in\N}\subseteq\LIP_\sfd(\X)\)
and \(\tilde g_n\leq\tilde g_{n+1}\leq f\) for all \(n\in\N\).
We claim that \(\tilde g_n(x)\to f(x)\) as \(n\to\infty\) for every \(x\in K\).
To prove it, fix \(x\in K\) and \(\delta>0\). Pick \(\bar n\in\N\) such that
\(1/\bar n\leq\delta\) and \(\sfd(x,x_{\bar n})\leq\delta\). Then for any
\(n\geq\bar n\) it holds that
\[
\tilde g_n(x)\geq -L\,\sfd(x,x_{\bar n})+f(x_{\bar n})-\frac{1}{n}
\geq f(x)-2L\,\sfd(x,x_{\bar n})-\frac{1}{\bar n}\geq f(x)-(2L+1)\delta,
\]
which grants that \(\tilde g_n(x)\nearrow f(x)\) by arbitrariness of \(\delta\).
Therefore, we have that \(\tilde g_n\to f\) uniformly on \(K\), so that there
exists \(n\in\N\) for which the function \(\tilde g\coloneqq\tilde g_n\) satisfies
\(|\tilde g-f|\leq\eps/2\) on \(K\). Given any \(i\in\N\), let us define the
function \(g_i\in\LIP_{\sfd_i}(\X)\) as
\[
g_i(x)\coloneqq\big(-L\,\sfd_i(x,x_1)+f(x_1)\big)\vee\dots\vee
\big(-L\,\sfd_i(x,x_n)+f(x_n)\big)-\frac{1}{n},\quad\text{ for every }x\in\X.
\]
Note that \(g_i\searrow\tilde g\) pointwise on \(K\), as a consequence
of the assumption \(\sfd_i\nearrow\sfd\). Since each \(g_i\) is
continuous with respect to \(\sfd\), we deduce that \(g_i\to\tilde g\)
uniformly on \(K\), thus for some \(i\in\N\) the function \(g\coloneqq g_i\)
satisfies \(|g-\tilde g|\leq\eps/2\) on \(K\). Hence, it holds that
\(|g-f|\leq\eps\) on \(K\), yielding \eqref{eq:approx_Lip_cl1}. Finally,
we have that \(\Lip_{\sfd_i}(g)\leq L=\Lip_\sfd(f)\) by construction,
whence \eqref{eq:approx_Lip_cl2} and accordingly the statement follow.
\end{proof}
By combining Lemma \ref{lem:approx_Lip} with a partition of
unity argument and the extension result in \cite{DMGP20},
we show that also the asymptotic slope can
be kept under control (in an integral sense).
\begin{proposition}\label{prop:approx_lip_a}
Let \((\X,\sfd,\mm)\) be a metric measure space.
Suppose to have a sequence \((\sfd_i)_{i\in\N}\) of distances on \(\X\)
such that \(\sfd_i(x,y)\nearrow\sfd(x,y)\) as \(i\to\infty\) for every \(x,y\in\X\)
and \(\tau(\sfd_i)=\tau(\sfd)\) for every \(i\in\N\). Fix an exponent
\(p\in(1,\infty)\) and a boundedly-supported function \(f\in\LIP_\sfd(\X)\).
Then for any \(\eps>0\) there exist \(i\in\N\) and \(g\in\LIP_{\sfd_i}(\X)\)
boundedly-supported such that
\begin{subequations}\begin{align}\label{eq:approx_lip_a_cl1}
\int|g-f|^p\,\d\mm&\leq\eps,\\
\label{eq:approx_lip_a_cl2}
\int\lip_a^{\sfd_i}(g)^p\,\d\mm&\leq\int\lip_a^\sfd(f)^p\,\d\mm+\eps.
\end{align}\end{subequations}
\end{proposition}
\begin{proof}
First of all, fix a point \(\bar x\in\X\) and a radius \(R>0\) such
that \({\rm spt}(f)\subseteq B_R^{\sfd_1}(\bar x)\).
Denote by \(B\) the ball \(\bar B_{R+2}^{\sfd_1}(\bar x)\).
Moreover, fix any \(\eps'\in(0,1/4)\) such that
\begin{equation}\label{eq:approx_lip_a_aux0}
\bigg[\Big(3p\,\Lip_\sfd(f)^{p-1}+1\Big)\mm(B)+
\Big(15\,\Lip_\sfd(f)+\sup_\X|f|+7\Big)^p\bigg]\eps'
\leq\eps.
\end{equation}
{\color{blue}\textsc{Step 1: Construction of the auxiliary function \(\tilde h\).}}
Since \(\X\ni x\mapsto\Lip_\sfd\big(f;B_{1/n}^\sfd(x)\big)\)
is a Borel function for any \(n\in\N\) and
\(\lip_a^\sfd(f)(x)=\lim_{n\to\infty}\Lip_\sfd\big(f;B_{1/n}^\sfd(x)\big)\)
for every \(x\in\X\), by virtue of Egorov's theorem there exist \(K\subseteq B\)
compact and \(r>0\) with \(\mm(B\setminus K)\leq\eps'\) and
\begin{equation}\label{eq:approx_lip_a_aux1}
\Lip_\sfd\big(f;B_{4r}^\sfd(x)\big)\leq\lip_a^\sfd(f)(x)+\eps',
\quad\text{ for every }x\in K.
\end{equation}
Choose some points \(x_1,\ldots,x_k\in K\) for which
\(K\subseteq\bigcup_{j=1}^k B_r^\sfd(x_j)\). Fix a \(\sfd_1\)-Lipschitz partition
of unity \(\{\psi_1,\ldots,\psi_k\}\) of \(K\) subordinated to
\(\big\{K\cap B_r^\sfd(x_1),\ldots,K\cap B_r^\sfd(x_k)\big\}\).
Videlicet, each function \(\psi_j\colon K\to[0,1]\) is \(\sfd_1\)-Lipschitz,
satisfies \({\rm spt}(\psi_j)\subseteq K\cap B_r^\sfd(x_j)\), and
\(\sum_{j=1}^k\psi_j(x)=1\) for every \(x\in K\). Since \(\sfd_i\to\sfd\)
uniformly on \(K\times K\) (by Remark \ref{rmk:unif_conv_cpt}), there exists
\(i_0\in\N\) such that \(\sfd(x,y)\leq\sfd_i(x,y)+\eps'r\) for every
\(x,y\in K\) and \(i\geq i_0\). Given any \(j=1,\ldots,k\), pick some function
\(f_j\in\LIP_\sfd(\X)\) such that \(f_j|_{B_{2r}^\sfd(x_j)}=f|_{B_{2r}^\sfd(x_j)}\)
and \(\Lip_\sfd(f_j)=\Lip_\sfd\big(f;B_{2r}^\sfd(x_j)\big)\), thus we can find
(by Lemma \ref{lem:approx_Lip}) an index \(i(j)\geq i_0\) and a function
\(h_j\in\LIP_{\sfd_{i(j)}}(\X)\) such that
\begin{subequations}\begin{align}\label{eq:approx_lip_a_aux2}
\big|h_j(x)-f_j(x)\big|&\leq\frac{\eps'}{\big(k\,\Lip_{\sfd_1}(\psi_j)\big)\vee 1},
\quad\text{ for every }x\in K,\\
\label{eq:approx_lip_a_aux3}
\Lip_{\sfd_{i(j)}}(h_j)&\leq\Lip_\sfd(f_j)=\Lip_\sfd\big(f;B_{2r}^\sfd(x_j)\big).
\end{align}\end{subequations}
Let us denote \(i\coloneqq\max\big\{i(1),\ldots,i(k)\big\}\) and
\(\tilde\sfd\coloneqq\sfd_i|_{K\times K}\). Moreover, we define
\(\tilde h\colon K\to\R\) as
\[
\tilde h(x)\coloneqq\sum_{j=1}^k\psi_j(x)\,h_j(x),\quad\text{ for every }x\in K.
\]
{\color{blue}\textsc{Step 2: Estimates for the Lipschitz constant of \(\tilde h\).}}
We claim that \(\tilde h\in\LIP_{\tilde\sfd}(K)\) and that
\(\Lip_{\tilde\sfd}(\tilde h)\leq\eps'+5\,\Lip_\sfd(f)\).
In order to prove it, fix any \(y,z\in K\). Then we have that
\begin{equation}\begin{split}\label{eq:approx_lip_a_aux4}
\big|\tilde h(y)-\tilde h(z)\big|&\leq\bigg|\sum_{j=1}^k\psi_j(y)
\big(h_j(y)-h_j(z)\big)\bigg|+\bigg|\sum_{j=1}^k\big(\psi_j(y)-\psi_j(z)\big)
\big(h_j(z)-f(z)\big)\bigg|\\
&\leq\sum_{j=1}^k\psi_j(y)\big|h_j(y)-h_j(z)\big|+
\sum_{j=1}^k\big|\psi_j(y)-\psi_j(z)\big|\big|h_j(z)-f(z)\big|.
\end{split}\end{equation}
Observe that the first term in the second line of the above formula can be estimated as
\begin{equation}\label{eq:approx_lip_a_aux5}
\sum_{j=1}^k\psi_j(y)\big|h_j(y)-h_j(z)\big|\leq
\sum_{j=1}^k\psi_j(y)\,\Lip_{\sfd_{i(j)}}(h_j)\,\sfd_{i(j)}(y,z)
\overset{\eqref{eq:approx_lip_a_aux3}}\leq\Lip_\sfd(f)\,\sfd_i(y,z).
\end{equation}
In order to estimate the second term in \eqref{eq:approx_lip_a_aux4},
fix \(j=1,\ldots,k\). We consider three cases:
\begin{itemize}
\item[\(\rm i)\)] If \(z\in B_{2r}^\sfd(x_j)\), then \(f(z)=f_j(z)\) and
accordingly
\[
\big|\psi_j(y)-\psi_j(z)\big|\big|h_j(z)-f(z)\big|
\overset{\eqref{eq:approx_lip_a_aux2}}\leq
\Lip_{\sfd_i}(\psi_j)\,\sfd_i(y,z)\,\frac{\eps'}{k\,\Lip_{\sfd_1}(\psi_j)}
\leq\frac{\eps'}{k}\,\sfd_i(y,z).
\]
\item[\(\rm ii)\)] If \(z\notin B_{2r}^\sfd(x_j)\) and \(y\in B_r^\sfd(x_j)\),
then \(f_j(y)=f(y)\) and \(\sfd(y,z)>r\). In particular,
\begin{equation}\label{eq:approx_lip_a_aux5bis}
\frac{\sfd(y,z)}{\sfd_i(y,z)}\leq\frac{\sfd_i(y,z)+\eps' r}{\sfd_i(y,z)}
\leq 1+\frac{\eps' r}{\sfd(y,z)-\eps' r}<1+\frac{\eps'}{1-\eps'}<2,
\end{equation}
whence it follows that
\[\begin{split}
&\big|\psi_j(y)-\psi_j(z)\big|\big|h_j(z)-f(z)\big|\\
\overset{\phantom{\eqref{eq:approx_lip_a_aux3}}}\leq\,
&\big|\psi_j(y)-\psi_j(z)\big|\Big[\big|h_j(z)-f_j(z)\big|
+\big|f_j(z)-f_j(y)\big|+\big|f_j(y)-f(z)\big|\Big]\\
\overset{\phantom{\eqref{eq:approx_lip_a_aux3}}}\leq\,
&\Lip_{\sfd_i}(\psi_j)\,\sfd_i(y,z)\,
\big|h_j(z)-f_j(z)\big|+\psi_j(y)\,\big|f_j(z)-f_j(y)\big|
+\psi_j(y)\,\big|f(y)-f(z)\big|\\
\overset{\eqref{eq:approx_lip_a_aux3}}\leq\,
&\Lip_{\sfd_i}(\psi_j)\,\sfd_i(y,z)\,\frac{\eps'}
{k\,\Lip_{\sfd_1}(\psi_j)}+\psi_j(y)\,\Lip_\sfd(f_j)\,\sfd(y,z)
+\psi_j(y)\,\Lip_\sfd(f)\,\sfd(y,z)\\
\overset{\phantom{\eqref{eq:approx_lip_a_aux3}}}\leq\,
&\frac{\eps'}{k}\,\sfd_i(y,z)+2\,\psi_j(y)\,\Lip_\sfd(f)\,\sfd(y,z)
\overset{\eqref{eq:approx_lip_a_aux5bis}}\leq
\bigg(\frac{\eps'}{k}+4\,\psi_j(y)\,\Lip_\sfd(f)\bigg)\sfd_i(y,z).
\end{split}\]
\item[\(\rm iii)\)] If \(z\notin B_{2r}^\sfd(x_j)\) and \(y\notin B_r^\sfd(x_j)\),
then trivially \(\big|\psi_j(y)-\psi_j(z)\big|\big|h_j(z)-f(z)\big|=0\).
\end{itemize}
By combining the estimates we obtained in \(\rm i)\), \(\rm ii)\), \(\rm iii)\)
with \eqref{eq:approx_lip_a_aux5} and \eqref{eq:approx_lip_a_aux4}, we deduce that
\[
\big|\tilde h(y)-\tilde h(z)\big|\leq\big(\eps'+5\,\Lip_\sfd(f)\big)\sfd_i(y,z),
\quad\text{ for every }y,z\in K.
\]
This proves that \(\tilde h\in\LIP_{\tilde\sfd}(K)\) and
\(\Lip_{\tilde\sfd}(\tilde h)\leq\eps'+5\,\Lip_\sfd(f)\),
yielding the sought conclusion.\\
{\color{blue}\textsc{Step 3: Estimates for the asymptotic slope of \(\tilde h\).}}
Next we claim that
\begin{equation}\label{eq:approx_lip_a_aux6}
\lip_a^{\tilde\sfd}(\tilde h)(x)\leq\lip_a^\sfd(f)(x)+2\eps',\quad\text{ for every }x\in K.
\end{equation}
To prove it, fix any \(\delta<\eps' r\) and \(y,z\in B_\delta^{\tilde\sfd}(x)\).
Define \(F\coloneqq\big\{j=1,\ldots,k\,:\,\sfd(x,x_j)<3r/2\big\}\).
If \(j\notin F\), then \(y,z\notin B_r^\sfd(x_j)\) and thus
\(\psi_j(y)=\psi_j(z)=0\), as it is granted by the estimates
\[
\sfd(y,x_j)\geq\sfd(x,x_j)-\sfd(x,y)\geq\frac{3r}{2}-\sfd_i(x,y)-\eps' r
>\bigg(\frac{3}{2}-\eps'\bigg)r-\delta>\bigg(\frac{3}{2}-2\eps'\bigg)r>r,
\]
and similarly for \(\sfd(z,x_j)\). If \(j\in F\), then \(B_{2r}^\sfd(x_j)
\subseteq B_{4r}^\sfd(x)\) and \(f_j(z)=f(z)\). The latter claim follows from
the fact that \(z\in B_{2r}^\sfd(x_j)\), which is granted by the estimates
\[
\sfd(z,x_j)\leq\sfd(z,x)+\sfd(x,x_j)<\sfd_i(z,x)+\eps' r+\frac{3r}{2}
<\delta+\bigg(\eps'+\frac{3}{2}\bigg)r<\bigg(2\eps'+\frac{3}{2}\bigg)r<2r.
\]
Therefore, by using \eqref{eq:approx_lip_a_aux4} and the above considerations,
we obtain that
\[\begin{split}
\big|\tilde h(y)-\tilde h(z)\big|&\overset{\eqref{eq:approx_lip_a_aux2}}\leq
\sum_{j\in F}\psi_j(y)\,\Lip_{\sfd_{i(j)}}(h_j)\,\sfd_{i(j)}(y,z)+
\sum_{j\in F}\Lip_{\sfd_i}(\psi_j)\,\sfd_i(y,z)\,\frac{\eps'}{k\,\Lip_{\sfd_1}(\psi_j)}\\
&\overset{\eqref{eq:approx_lip_a_aux3}}\leq
\bigg[\sum_{j\in F}\psi_j(y)\,\Lip_\sfd\big(f;B_{2r}^\sfd(x_j)\big)
+\eps'\bigg]\sfd_i(y,z)\\
&\overset{\phantom{\eqref{eq:approx_lip_a_aux3}}}\leq
\Big[\Lip_\sfd\big(f;B_{4r}^\sfd(x)\big)+\eps'\Big]\sfd_i(y,z)
\overset{\eqref{eq:approx_lip_a_aux1}}\leq\big[\lip_a^\sfd(f)(x)+2\eps'\big]\sfd_i(y,z).
\end{split}\]
Thanks to the arbitrariness of \(y,z\in B_\delta^{\tilde\sfd}(x)\), we deduce
that \(\Lip_{\tilde\sfd}\big(\tilde h;B_\delta^{\tilde\sfd}(x)\big)\leq
\lip_a^\sfd(f)(x)+2\eps'\), whence by letting \(\delta\searrow 0\) we can
finally conclude that the inequality in \eqref{eq:approx_lip_a_aux6} is verified.\\
{\color{blue}\textsc{Step 4: Construction of the function \(g\).}}
Given any point \(x\in K\), it holds that
\begin{equation}\label{eq:approx_lip_a_aux7}
\big|\tilde h(x)-f(x)\big|\leq\sum_{j=1}^k\psi_j(x)\big|h_j(x)-f(x)\big|=\sum_{j=1}^k
\psi_j(x)\big|h_j(x)-f_j(x)\big|\overset{\eqref{eq:approx_lip_a_aux2}}\leq\eps'.
\end{equation}
In particular, we have that \(\sup_K|\tilde h|\leq\sup_\X|f|+1\). Recall also that
\(\Lip_{\tilde\sfd}(\tilde h)\leq 5\,\Lip_\sfd(f)+\eps'\), as proven in
\textsc{Step 2}. Therefore, by applying \cite[Theorem 1.1]{DMGP20}
we can find a function \(h\in\LIP_{\sfd_i}(\X)\) with
\(h|_K=\tilde h\) such that
\(\lip_a^{\sfd_i}(h)(x)=\lip_a^{\tilde\sfd}(\tilde h)(x)\)
for every \(x\in K\) and
\begin{equation}\label{eq:approx_lip_a_aux8}
\Lip_{\sfd_i}(h)\leq\Lip_{\tilde\sfd}(\tilde h)+\eps'
\leq 5\,\Lip_\sfd(f)+2\eps'\eqqcolon C.
\end{equation}
Define \(G\coloneqq\big\{x\in\X\,:\,\sfd_i(x,{\rm spt}(f)\cap K)\leq 2\big\}\)
and observe that \(\sup_G|h|\leq 2C+\sup_\X|f|+1\).
Indeed, given any point \(x\in G\), one has that
\[\begin{split}
\big|h(x)\big|&
\overset{\phantom{\eqref{eq:approx_lip_a_aux8}}}\leq
\inf_{y\in{\rm spt}(f)\cap K}\Big[\big|h(x)-h(y)\big|+\big|h(y)\big|\Big]\leq
\Lip_{\sfd_i}(h)\inf_{y\in{\rm spt}(f)\cap K}\sfd_i(x,y)+\sup_K|h|\\
&\overset{\eqref{eq:approx_lip_a_aux8}}\leq
C\,\sfd_i(x,{\rm spt}(f)\cap K)+\sup_K|\tilde h|\leq 2C+\sup_\X|f|+1.
\end{split}\]
Moreover, we have that \(G\subseteq B=\bar B_{R+2}^{\sfd_1}(\bar x)\).
Indeed, by using that \({\rm spt}(f)\subseteq B_R^{\sfd_1}(\bar x)\), we get
\[
\sfd_1(x,\bar x)\leq\inf_{y\in{\rm spt}(f)\cap K}\big[\sfd_1(x,y)
+\sfd_1(y,\bar x)\big]\leq\inf_{y\in{\rm spt}(f)\cap K}\sfd_i(x,y)+R
\leq R+2,
\]
for every \(x\in G\). Let us now define the
\(\sfd_i\)-Lipschitz cut-off function \(\eta\colon\X\to[0,1]\) as
\[
\eta(x)\coloneqq\Big(\big(2-\sfd_i(x,{\rm spt}(f)\cap K)\big)
\wedge 1\Big)\vee 0,\quad\text{ for every }x\in\X.
\]
It holds that \(\eta=1\) on a neighbourhood of \({\rm spt}(f)\cap K\)
and that \(\Lip_{\sfd_i}(\eta)\leq 1\). Given that \(\eta=0\) in
\(\X\setminus G\), it also holds that \({\rm spt}(\eta)\subseteq G\).
We then define the function \(g\colon\X\to\R\) as
\(g\coloneqq\eta h\).\\
{\color{blue}\textsc{Step 5: Conclusion.}}
Note that \(g\in\LIP_{\sfd_i}(\X)\), \({\rm spt}(g)\subseteq G\),
and \(\sup_\X|g|\leq 2C+\sup_\X|f|+1\). Let us estimate
\(\Lip_{\sfd_i}(g)\). Since \(\big|g(x)-g(y)\big|\leq\eta(x)
\big|h(x)-h(y)\big|+\big|\eta(x)-\eta(y)\big|\big|h(y)\big|\) holds
for every \(x,y\in\X\), we obtain that \(\big|g(x)-g(y)\big|\leq
\big(C+\sup_G|h|\big)\sfd_i(x,y)\) whenever \(y\in G\), whence it
follows that \(\Lip_{\sfd_i}(g)\leq 3C+\sup_\X|f|+1\).
The same computations give
\begin{equation}\label{eq:approx_lip_a_aux9}
\Lip_{\sfd_i}(g;E)\leq\Lip_{\sfd_i}(h;E)+\sup_E|h|,
\quad\text{ for every }E\subseteq\X.
\end{equation}
On the one hand, since \(g\) and \(h\) agree on a neighbourhood
of \({\rm spt}(f)\cap K\), for any \(x\in{\rm spt}(f)\cap K\)
we have that \(\big|g(x)-f(x)\big|\leq\eps'\) by \eqref{eq:approx_lip_a_aux7}
and \(\lip_a^{\sfd_i}(g)(x)\leq\lip_a^\sfd(f)(x)+2\eps'\)
by \eqref{eq:approx_lip_a_aux6}. On the other hand, if
\(x\in K\setminus{\rm spt}(f)\), then \(f(x)=\lip_a^\sfd(f)(x)=0\),
thus accordingly we can deduce from \eqref{eq:approx_lip_a_aux7} that
\(\big|g(x)-f(x)\big|=\eta(x)\big|h(x)\big|\leq\eps'\),
while \eqref{eq:approx_lip_a_aux6}, \eqref{eq:approx_lip_a_aux7},
and \eqref{eq:approx_lip_a_aux9} ensure that
\[\begin{split}
\lip_a^{\sfd_i}(g)(x)&=\lim_{\delta\searrow 0}\Lip_{\sfd_i}
\big(g;B_\delta^{\sfd_i}(x)\big)\leq\lim_{\delta\searrow 0}
\Lip_{\sfd_i}\big(h;B_\delta^{\sfd_i}(x)\big)+
\lim_{\delta\searrow 0}\sup_{B_\delta^{\sfd_i}(x)}|h|\\
&=\lip_a^{\sfd_i}(h)(x)+\big|h(x)\big|\leq 3\eps'.
\end{split}\]
All in all, we have shown that
\begin{subequations}\begin{align}\label{eq:approx_lip_a_aux10a}
\big|g(x)-f(x)\big|&\leq\left\{\begin{array}{ll}
\eps',\\
2C+\sup_\X|f|+1,
\end{array}\quad\begin{array}{ll}
\text{ if }x\in K,\\
\text{ if }x\in\X\setminus K,
\end{array}\right.\\\label{eq:approx_lip_a_aux10b}
\lip_a^{\sfd_i}(g)(x)&\leq\left\{\begin{array}{ll}
\lip_a^\sfd(f)(x)+3\eps',\\
3C+\sup_\X|f|+1,
\end{array}\quad\begin{array}{ll}
\text{ if }x\in K,\\
\text{ if }x\in\X\setminus K.
\end{array}\right.
\end{align}\end{subequations}
It remains to check that \(g\) satisfies \eqref{eq:approx_lip_a_cl1}
and \eqref{eq:approx_lip_a_cl2}. Recall that
\({\rm spt}(f),{\rm spt}(g)\subseteq B\). Then
\[\begin{split}
\int|g-f|^p\,\d\mm&
\overset{\phantom{\eqref{eq:approx_lip_a_aux10a}}}=
\int_K|g-f|^p\,\d\mm+\int_{B\setminus K}|g-f|^p\,\d\mm\\
&\overset{\eqref{eq:approx_lip_a_aux10a}}\leq
\mm(K)\,(\eps')^p+\mm(B\setminus K)\Big(2C+\sup_\X|f|+1\Big)^p\\
&\overset{\phantom{\eqref{eq:approx_lip_a_aux10a}}}\leq
\bigg[\mm(B)+\Big(2C+\sup_\X|f|+1\Big)^p\bigg]\eps'.
\end{split}\]
Moreover, it holds that
\[\begin{split}
\int\lip_a^{\sfd_i}(g)^p\,\d\mm
&\overset{\phantom{\eqref{eq:approx_lip_a_aux10b}}}=
\int_K\lip_a^{\sfd_i}(g)^p\,\d\mm+\int_{B\setminus K}
\lip_a^{\sfd_i}(g)^p\,\d\mm\\
&\overset{\eqref{eq:approx_lip_a_aux10b}}\leq
\int_K\big(\lip_a^\sfd(f)+3\eps'\big)^p\,\d\mm+
\mm(B\setminus K)\Big(3C+\sup_\X|f|+1\Big)^p\\
&\overset{\phantom{\eqref{eq:approx_lip_a_aux10b}}}\leq
\int\lip_a^\sfd(f)^p\,\d\mm+3p\eps'\int_B\lip_a^\sfd(f)^{p-1}\,\d\mm
+\Big(3C+\sup_\X|f|+1\Big)^p\eps'\\
&\overset{\phantom{\eqref{eq:approx_lip_a_aux10b}}}\leq
\int\lip_a^\sfd(f)^p\,\d\mm+
\bigg[3p\,\Lip_\sfd(f)^{p-1}\mm(B)+\Big(3C+\sup_\X|f|+1\Big)^p\bigg]\eps'.
\end{split}\]
By taking \eqref{eq:approx_lip_a_aux0} into account, we can finally conclude
that \eqref{eq:approx_lip_a_cl1} and \eqref{eq:approx_lip_a_cl2} are verified.
\end{proof}
\section{Mosco-convergence of Cheeger energies}
By applying Proposition \ref{prop:approx_lip_a},
we can easily obtain our main \(\Gamma\)-convergence result.
\begin{theorem}\label{thm:main_Mosco-conv}
Let \((\X,\sfd,\mm)\) be a metric measure space. Let \((\sfd_i)_{i\in\N}\)
be a sequence of complete distances on \(\X\) such that \(\sfd_i\nearrow\sfd\)
as \(i\to\infty\). Suppose \(\tau(\sfd_i)=\tau(\sfd)\) for all \(i\in\N\).
Fix \(p\in(1,\infty)\). Then \(\E_{\Ch,p}^{\sfd_i}\) Mosco-converges
to \(\E_{\Ch,p}^\sfd\) as \(i\to\infty\). Videlicet, the following
properties hold:
\begin{itemize}
\item[\(\rm i)\)] \textsc{Weak \(\Gamma\)-lim inf.}
If \((f_i)_{i\in\N}\subseteq L^p(\mm)\) weakly converges to \(f\in L^p(\mm)\),
then it holds
\[
\E_{\Ch,p}^\sfd(f)\leq\limi_{i\to\infty}\E_{\Ch,p}^{\sfd_i}(f_i).
\]
\item[\(\rm ii)\)] \textsc{Strong \(\Gamma\)-lim sup.}
Given any \(f\in L^p(\mm)\), there exists a sequence
\((f_i)_{i\in\N}\subseteq L^p(\mm)\) that strongly converges to \(f\) and satisfies
\[
\E_{\Ch,p}^\sfd(f)\geq\lims_{i\to\infty}\E_{\Ch,p}^{\sfd_i}(f_i).
\]
\end{itemize}
\end{theorem}
\begin{proof}
Item \(\rm i)\) can be easily proven: given any \(f\in L^p(\mm)\) and
\((f_i)_{i\in\N}\subseteq L^p(\mm)\) with \(f_i\rightharpoonup f\)
weakly in \(L^p(\mm)\), the weak lower semicontinuity of
\(\E_{\Ch,p}^\sfd\colon L^p(\mm)\to[0,+\infty]\) grants that
\[
\E_{\Ch,p}^\sfd(f)\leq\limi_{i\to\infty}\E_{\Ch,p}^\sfd(f_i)
\overset{\eqref{eq:ineq_ener}}\leq\limi_{i\to\infty}\E_{\Ch,p}^{\sfd_i}(f_i).
\]
Let us then pass to the verification of item \(\rm ii)\). Let \(f\in L^p(\mm)\)
be given. If \(f\notin W^{1,p}(\X)\), then \(\E_{\Ch,p}^\sfd(f)=+\infty\)
and accordingly the \(\Gamma\)-lim sup inequality is trivially verified (by
taking, for instance, \(f_i\coloneqq f\) for every \(i\in\N\)). Now suppose
\(f\in W^{1,p}(\X)\). By definition of \(\E_{\Ch,p}^\sfd\),
we can find a sequence
\((\tilde f_n)_n\subseteq\LIP_\sfd(\X)\) of boundedly-supported
functions such that \(\tilde f_n\to f\) strongly in \(L^p(\mm)\)
and \(\E_{\Ch,p}^\sfd(f)=\lim_n\E_{a,p}^\sfd(\tilde f_n)\).
By Proposition \ref{prop:approx_lip_a}, we can find
\(\iota\colon\N\to\N\) increasing and a sequence \((g_n)_n\) of boundedly-supported
functions \(g_n\in\LIP_{\sfd_{\iota(n)}}(\X)\) such that
\[
\int|g_n-\tilde f_n|^p\,\d\mm\leq\frac{1}{n},\qquad
\E_{a,p}^{\sfd_{\iota(n)}}(g_n)\leq\E_{a,p}^\sfd(\tilde f_n)+\frac{1}{n}.
\]
In particular, \(g_n\to f\) strongly in \(L^p(\mm)\) and
\(\E_{\Ch,p}^\sfd(f)\geq\lims_n\E_{a,p}^{\sfd_{\iota(n)}}(g_n)
\geq\lims_n\E_{\Ch,p}^{\sfd_{\iota(n)}}(g_n)\). Finally, we define
the recovery sequence \((f_i)_i\subseteq L^p(\mm)\) in the following way:
\[
f_i\coloneqq g_n,\quad\text{ for every }n\in\N\text{ and }
i\in\big\{\iota(n),\ldots,\iota(n+1)-1\big\}.
\]
Notice that \(f_i\to f\) strongly in \(L^p(\mm)\). Moreover, Remark
\ref{rmk:ineq_ener} grants that \(\E_{\Ch,p}^{\sfd_i}(f_i)\leq
\E_{\Ch,p}^{\sfd_{\iota(n)}}(g_n)\) whenever \(\iota(n)\leq i<\iota(n+1)\),
which implies that \(\lims_i\E_{\Ch,p}^{\sfd_i}(f_i)=
\lims_n\E_{\Ch,p}^{\sfd_{\iota(n)}}(g_n)\leq\E_{\Ch,p}^\sfd(f)\).
This gives the \(\Gamma\)-lim sup inequality, thus accordingly
the statement is achieved.
\end{proof}
It readily follows from Theorem \ref{thm:main_Mosco-conv} that
the infinitesimal Hilbertianity condition is stable under taking
increasing limits of the distances (while keeping the measure fixed).
Videlicet:
\begin{corollary}\label{cor:stab_iH}
Let \((\X,\sfd,\mm)\) be a metric measure space. Let \((\sfd_i)_{i\in\N}\)
be a sequence of complete distances on \(\X\) such that \(\sfd_i\nearrow\sfd\) as
\(i\to\infty\) and \(\tau(\sfd_i)=\tau(\sfd)\) for every \(i\in\N\).
Suppose \((\X,\sfd_i,\mm)\) is infinitesimally Hilbertian for every \(i\in\N\).
Then \((\X,\sfd,\mm)\) is infinitesimally Hilbertian.
\end{corollary}
\begin{proof}
Theorem \ref{thm:main_Mosco-conv} implies that
\(\E_{\Ch,2}^{\sfd_i}\overset\Gamma\to\E_{\Ch,2}^\sfd\)
with respect to the strong topology of \(L^2(\mm)\),
thus \cite[Theorem 11.10]{DalMaso93} grants that
\(\E_{\Ch,2}^\sfd\) is a quadratic form, which gives
the statement.
\end{proof}
\begin{remark}\label{rmk:sub-Riem}{\rm
Let \(({\rm M},\sfd)\) be (the metric space associated with)
a generalised sub-Riemannian manifold, in the sense of
\cite[Definition 4.1]{LDLP19}. Then there exists a sequence
\((\sfd_i)_{i\in\N}\) of distances on \(\rm M\), induced by
Riemannian metrics, such that \(\sfd_i\nearrow\sfd\);
cf.\ \cite[Corollary 5.2]{LDLP19}. Suppose \(\sfd\) and each
\(\sfd_i\) are complete distances. Fix a Radon measure \(\mm\)
on \(\rm M\). Then \cite[Theorem 4.11]{LP20} ensures that each
\(({\rm M},\sfd_i,\mm)\) is infinitesimally Hilbertian.
Therefore, by applying Corollary \ref{cor:stab_iH} we can
conclude that \(({\rm M},\sfd,\mm)\) is infinitesimally
Hilbertian as well. This argument provides an alternative
proof of \cite[Corollary 5.6]{LDLP19}.
\fr}\end{remark}

We conclude the paper by illustrating an example which shows that the
results of this section cannot hold if the assumption of monotone
convergence from below of the distances is replaced by a monotone
convergence from above.
\begin{example}\label{ex:failure_conv_from_above}{\rm
Let \((\X,\sfd,\mm)\) be any metric measure space such that \(\sfd\leq 1\).
Given any \(i\in\N\), we define the `snowflake' distance \(\sfd_i\)
on \(\X\) as \(\sfd_i(x,y)\coloneqq\sfd(x,y)^{1-\frac{1}{i}}\) for
every \(x,y\in\X\). Then we have \(\sfd_i(x,y)\searrow\sfd(x,y)\)
as \(i\to\infty\) for all \(x,y\in\X\) and \(\tau(\sfd_i)=\tau(\sfd)\)
for all \(i\in\N\). Since absolutely continuous
curves in \((\X,\sfd_i)\) are constant, it follows from
the results in  \cite{AmbrosioGigliSavare11-3} that
\[
\E_{\Ch,p}^{\sfd_i}(f)=0,\quad\text{ for every }p\in(1,\infty)
\text{ and }f\in L^p(\mm).
\]
In particular, each space \((\X,\sfd_i,\mm)\) is infinitesimally
Hilbertian. This shows that Theorem
\ref{thm:main_Mosco-conv} and Corollary \ref{cor:stab_iH}
might fail if we replace the assumption \(\sfd_i\nearrow\sfd\)
with \(\sfd_i\searrow\sfd\).
\fr}\end{example}
\def\cprime{$'$} \def\cprime{$'$}

\end{document}